\documentclass[11pt]{article}
\usepackage{amsmath}
\usepackage{amsthm,amsfonts,amssymb,latexsym}
\usepackage[utf8]{inputenc}
\usepackage[english]{babel}

\newtheorem{theorem}{Theorem}
\newtheorem{proposition}[theorem]{Proposition}
\newtheorem{lemma}[theorem]{Lemma}

\theoremstyle{definition}

\theoremstyle{remark}

\newtheorem{remark}{Remark}[theorem]
\def\R{{\mathbb R}}
\def\N{{\mathbb N}}
\newcommand{\xR}{{]}{-\infty},+\infty]}
\newcommand{\Rex}{\xR}
\newcommand{\Rb}{\overline{\R}}

\newcommand{\ps}{\smallbreak}

\newcommand{\lsc}{lsc}

\newcommand{\del}{\partial}

\newcommand{\delc}{\widehat{\del}}

\newcommand{\delsf}{f^\del}

\newcommand{\dom} {{\rm dom} \kern.15em}
\newcommand{\tq}{:}
\newcommand{\la}{\langle}
\newcommand{\ra}{\rangle}
\newcommand{\ld}{\lambda}
\newcommand{\eps}{\varepsilon}

\newcommand{\bx}{\bar{x}}
\newcommand{\xb}{\bar{x}}

\newcommand{\xt}{x_\eps}

\newcommand{\tow}{{\stackrel{w^*}{\longrightarrow}}\;}

\begin{document}
\thispagestyle{empty}
\begin{center}
{\large\bf\sc 
Subderivative-subdifferential duality formula}
\end{center}

\begin{center}
  {\small\begin{tabular}{c}
  Marc Lassonde\\
   Universit\'e des Antilles, BP 150, 97159 Pointe \`a Pitre, France; and\\
   LIMOS, Universit\'e Blaise Pascal, 63000 Clermont-Ferrand, France\\
  E-mail: marc.lassonde@gmail.com
  \end{tabular}}
\end{center}

\medbreak\noindent
\textbf{Abstract.}
We provide a formula linking the radial subderivative
to other subderivatives and subdifferentials for arbitrary extended
real-valued lower semicontinuous functions.
\medbreak\noindent
  
\medbreak\noindent
\section{Introduction}\label{intro}
Tyrrell Rockafellar and Roger Wets \cite[p. 298]{RW98} discussing the duality
between subderivatives and subdifferentials write
{\small
\begin{quote}
In the presence
of regularity, the subgradients and subderivatives of a function $f$
are completely dual to each other. [\ldots]
For functions $f$ that aren't subdifferentially regular, subderivatives and
subgradients can have distinct and independent roles, and
some of the duality must be relinquished.
\end{quote}
}

Jean-Paul Penot \cite[p. 263]{Pen13}, in the introduction to the chapter dealing
with elementary and viscosity subdifferentials, writes
{\small
\begin{quote}
In the present framework, in contrast to the convex
objects, the passages from directional derivatives (and tangent
cones) to subdifferentials (and normal cones, respectively)
are one-way routes, because
the first notions are nonconvex, while a dual object exhibits convexity properties.
\end{quote}
}
In the chapter concerning Clarke subdifferentials \cite[p. 357]{Pen13},
he notes
{\small
\begin{quote} 
In fact, in this theory, a complete primal-dual picture is available:
besides a normal cone concept, one
has a notion of tangent cone to a set, and besides a subdifferential
for a function
one has a notion of directional derivative. Moreover,
inherent convexity properties ensure a full duality between these notions. [\ldots].
These facts represent great theoretical and practical
advantages.
\end{quote}
}

In this paper, we consider arbitrary extended real-valued
lower semicontinuous functions and arbitrary subdifferentials.
In spite of the above quotes, we show that there is always a duality formula linking
the subderivatives and subdifferentials of such functions.
Moreover, we show that at points where the (lower semicontinuous)
function satisfies a mild regularity property
(called radial accessibility),
the upper radial subderivative is always a lower bound for
the expressions in the duality formula.

This lower bound is an equality in particular
for convex functions, but also for
various other classes of functions.
For such functions, the radial subderivative can  therefore be recovered
from the subdifferential, and consequently 
the function itself, up to a constant,  can be recovered from 
the subdifferential.
This issue is discussed elsewhere.

\section{Subderivatives}

In the sequel, $X$ is a real Banach space with unit ball $B_X$,
$X^*$ is its topological dual,
and $\la .,. \ra$ is the duality pairing.
For $x, y \in X$, we let $[x,y]:=\{ x+t(y-x) \tq t\in[0,1]\}$;
the sets $]x,y[$ and $[x,y[$ are defined accordingly.
Set-valued operators $T:X\rightrightarrows X^*$
are identified with their graph $T\subset X\times X^*$.
For a subset $A\subset X$, $x\in X$ and $\ld>0$, we let
$d_A(x):=\inf_{y\in A} \|x-y\|$ and $B_\ld(A):=\{ y\in X\tq d_A(y)\le \ld\}$. 
All extended-real-valued functions $f : X\to\xR$ are assumed to be
lower semicontinuous (lsc)
and \textit{proper}, which means that
the set  $\dom f:=\{x\in X\tq f(x)<\infty\}$ is non-empty.
\medbreak
For a lsc function $f:X\to\xR$, a point $\xb\in\dom f$ and
a direction $u\in X$,
we consider the following basic subderivatives (we essentially follow the
terminology of Penot's textbook \cite{Pen13}):

- the (lower right Dini) \textit{radial subderivative}: 
\begin{equation*}\label{Dinisub}
f^r(\xb;u):=\liminf_{t\searrow 0}\,\frac{f(\xb+tu)-f(\xb)}{t},
\end{equation*}
its upper version:
\begin{equation*}\label{Dinisub}
f^r_+(\xb;u):=\limsup_{t\searrow 0}\,\frac{f(\xb+tu)-f(\xb)}{t},
\end{equation*}
and its upper strict version (the \textit{Clarke subderivative}): 
\begin{equation*}\label{Clarkesub}
f^0(\bx;u):=  
\limsup_{t \searrow 0 \atop{(x,f(x)) \to (\bx,f(\bx))}}\frac{f(x+tu) -f(x)}{t};
\end{equation*}

- the (lower right Dini-Hadamard) \textit{directional subderivative}:
\begin{equation*}\label{Hsubderiv}
f^d(\bx;u):=
\liminf_{t \searrow 0 \atop{u' \to u}}\frac{f(\bx+tu')-f(\bx)}{t},
\end{equation*}
and its upper strict version (the Clarke-Rockafellar subderivative): 
\begin{equation*}\label{Csubderiv}
f^\uparrow(\bx;u):= \sup_{\delta>0} 
\limsup_{t \searrow 0 \atop{(x,f(x)) \to (\bx,f(\bx))}}
\inf_{u' \in B_{\delta}(u)}\frac{f(x+tu') -f(x)}{t}.
\end{equation*}

It is immediate from these definitions that the following inequalities hold
($\rightarrow$ means $\le$):
\begin{align*}
f^r(\xb;u)  & \rightarrow f^r_+(\xb;u)\rightarrow f^0(\xb;u)\\
\uparrow \quad &     \qquad\qquad\qquad\quad\uparrow\\
f^d(\xb;u)  & \qquad\longrightarrow \quad\quad f^\uparrow(\xb;u)
\end{align*}
It is well known (and easily seen) that for a function $f$
locally Lipschitz at $\xb$,
we have $f^r(\xb;u)=f^d(\xb;u)$ and $f^0(\xb;u)=f^\uparrow(\xb;u)$,
whereas for a lsc convex $f$, we have $f^d(\xb;u)=f^\uparrow(\xb;u)$.
A function $f$ satisfying such an equality is called \textit{regular}. 
However, in general,
$f^d(\xb;u)<f^\uparrow(\xb;u)$, and there are many other types of
subderivatives $f'$ which lie between $f^d$ and $f^\uparrow$.
\ps
The inequality stated in the theorem below is (much) less elementary. It
is the analytic form of Treiman's theorem \cite{Tre83}
on the inclusion of the lower limit of Boulingand contingent cones
at neighbouring points of $\xb$ into the Clarke tangent cone at $\xb$
in the context of a Banach space
(in finite dimensional spaces, equality holds between these objects,
as was shown earlier by Cornet \cite{Cor81} and Penot \cite{Pen81}).
A proof of this inequality (or equality in finite dimensional spaces)
based on this geometrical approach was given by Ioffe \cite{Iof84}
(see also Rockafellar-Wets \cite[Theorem 8.18]{RW98}).
For a proof (in the general context of Banach spaces)
using a multidirectional mean value inequality rather than the above
geometric approach, see Correa-Gajardo-Thibault \cite{CGT09}.

\begin{theorem}[Link between subderivatives]\label{Treiman}
Let $X$ be a Banach space,
$f:X\to\xR$ be lsc, $\xb\in\dom f$ and $u\in X$.
Then:
{
$$f^\uparrow(\xb;u)\le \sup_{\eps>0}\limsup_{x\to\xb} \inf_{u'\in B(u,\eps)}f^d(x;u').$$
}
\end{theorem}

\section{Subdifferentials}
Given a lsc function $f:X\to\xR$ and a point $\xb\in\dom f$,
we consider the following two basic subsets of the dual space $X^*$:

- the \textit{Moreau-Rockafellar subdifferential}
(the subdifferential of convex analysis):
\begin{equation*}\label{convex-sdiff}
 \del_{MR} f (\xb) :=
 \{ x^* \in X^* \tq \la x^*,y-\xb\ra + f(\xb) \leq f(y),\, \forall y \in X \};
\end{equation*}

- the \textit{Clarke subdifferential}, associated to the Clarke-Rockafellar subderivative:
\begin{eqnarray*}\label{Csub}
\partial_{C} f(\bx) := \{x^* \in X^* \tq \langle x^*,u\rangle \leq
f^\uparrow(\bx;u), \, \forall u \in X\}.
\end{eqnarray*}

All the classical subdifferentials (proximal,
Fr\'echet, Hadamard, Ioffe, Michel-Penot, \ldots)
lie between these two subsets.
It is well known that for a lsc convex $f$, $\del_{MR} f =\del_C f$,
so all the classical subdifferentials coincide in this case.
\ps
In the sequel, we call \textit{subdifferential} any operator $\del$
that associates 
a set-valued mapping $\partial f: X \rightrightarrows X^\ast$
to each function $f$ on $X$ so that
\begin{equation*}\label{inclusdansClarke}
\del_{MR} f\subset \partial f\subset \partial_{C} f
\end{equation*}
and the following \textit{Separation Principle} is satisfied in $X$:
\medbreak
(SP)
\textit{For any lsc $f,\varphi$ with $\varphi$ convex Lipschitz
near $\xb\in\dom f $,
if $f+\varphi$ admits a local minimum at $\xb$, then
$0\in \delc f(\xb)+ \del \varphi(\xb),$ where
\begin{multline}\label{wclosure}
\delc f(\xb):= \{\, \xb^*\in X^*\tq \mbox{there is a net }((x_\nu,x^*_\nu))_\nu\subset \del f \mbox{ with }\\
       (x_\nu,f(x_\nu))\to (\bx,f(\bx)),\ x^*_\nu\tow \bx^*,\ \limsup_\nu\,\la x^*_\nu,x_\nu-\xb\ra\le 0\,\}.
\end{multline}
}
\begin{remark} \label{SP-rem}
(a) In our paper \cite{JL13}, the set $\delc f(\xb)$ defined in
\eqref{wclosure} is called the \textit{weak*-controlled closure} of
the set-valued map $\del f$ at point $\xb$. The reason to consider
such a closure is that, even for a convex lsc function $f$, the
a priori simpler $strong\times weak^*$-closure of the graph of
$\del f=\del_{MR} f$ is too big for the Separation Principle to be meaningful.
The graph of $\del_{MR} f$ is not $strong\times weak^*$-closed in general:
see, e.g., \cite{JL02} for a discussion on what would be sufficient
to add to the $strong\times weak^*$ topology on $X\times X^*$
to guarantee the closure of such graphs.
More precisely, the graph of the convex subdifferential 
is $strong\times weak^*$-closed for each lsc convex function
if and only if X is finite dimensional (see \cite{BFG03}).
It is worth noting (and easily seen) that, as expected, always
$\del_{MR} f=\delc_{MR} f$.

\ps
(b) If we require the net
$((x_\nu,x^*_\nu))_\nu\subset\del f$ in \eqref{wclosure}
to be actually a sequence $((x_{n},x_{n}^*))_n$, $n\in\N$
(in which case the control assertion $\limsup_n\,\la x^*_n,x_n-\xb\ra\le 0$
is automatically satisfied), we obtain the so-called `limiting subdifferentials'.
A widely used such limiting subdifferential
is the
weak$^*$ sequential closure of the Fr\'echet subdifferential,
known as the \textit{Mordukhovich subdifferential}.

\ps
(c) The Separation Principle (SP) is a very simple property expected to
be satisfied by a subdifferential $\del$ in a Banach space $X$.
This property is actually equivalent to various
other properties of the
subdifferential $\del$ in the Banach space X: see \cite{JL13}.
\end{remark}

We recall that the Clarke subdifferential, the Michel-Penot subdifferential and
the Ioffe subdifferential satisfy the Separation Principle in any Banach space.
The elementary subdifferentials (proximal, Fr\'echet, Hadamard, \ldots),
as well as their viscosity and limiting versions,
satisfy the Separation Principle in appropriate Banach spaces:
the Fr\'echet subdifferential in Asplund spaces,
the Hadamard subdifferential in separable spaces,
the proximal subdifferential in Hilbert spaces.
The Moreau-Rockafellar subdifferential does not satisfy the Separation
Principle for the whole class of lsc (non necessarily convex) functions:
it is not a subdifferential for this wide class.
See, e.g. \cite{Iof12,JL13,Pen13} and the references therein.
\medbreak
The following link between the radial subderivative and
arbitrary subdifferentials was established in \cite[Theorem 2.1]{JL14}
(see also \cite[Theorem 3.2]{JL13}):

\begin{theorem}[Link between radial subderivative and
subdifferentials]\label{JL}
Let $X$ be a Banach space,
$f:X\to\Rex$ be lsc, $\xb\in\dom f$ and $u\in X$.
Then, there is a sequence $((x_n,x^*_n))\subset\del f$ such that
$x_n\to \xb$, $f(x_n)\to f(\xb)$,
$$f^r(\xb;u)\le \liminf_{n}\,\langle x^*_n,u\rangle
\text{ and }
\limsup_{n}\,\langle x^*_n,x_n-\xb\rangle\le 0.$$
\end{theorem}

\section{Subderivative-subdifferential duality formula}
A sequence $(x_n)\subset X$ is said to be
\textit{directionally convergent to $\xb$ from the direction $v\in X$},
written $x_n\to_v \xb$,
if there are two sequences $t_n\searrow 0$ (that is, $t_n\to 0$ with $t_n>0$)
and $v_n\to v$ such that
$x_n=\xb + t_n v_n$ for all $n$; equivalently: for every $\eps>0$
the sequence $(x_n-\xb)$ eventually lies in the open drop
${}]0,\eps B(v,\eps){[}:=\{\, tv'\tq 0<t<\eps,\ v'\in B(v,\eps)\,\}$.
Observe that for $v=0$, ${}]0,\eps B(v,\eps){[}=B(0,\eps^2)$
so $x_n\to_v\xb$ simply means $x_n\to\xb$.
We let $D(\xb,v,\eps):= \xb+{}]0,\eps B(v,\eps){[}$.

We call \textit{subderivative associated to a subdifferential $\del f$}
at a point $(\xb,u)\in \dom f\times X$, the \textit{support function} of
the set $\del f(\xb)$ in the direction $u$,
which we denote by
$$
f^\del(\xb;u):=
\sup \,\{\la \xb^*,u \ra \tq \xb^*\in\del f(\xb)\}.$$
In the theorem below, 
given a function $f:X\to\Rex$, we denote by $f':\dom f\times X\to \Rb$
any function lying between the subderivatives $f^d$ and $f^\uparrow$, that is: 
\begin{equation*}
f^d\le f'\le f^\uparrow.
\end{equation*}
Subderivatives and subdifferentials are linked by the following formula:

\begin{theorem}[Subderivative-subdifferential duality formula]\label{formula}
Let $X$ be a Banach space,
$f:X\to\xR$ be lsc, $\xb\in\dom f$ and $u\in X$.
Then, for any direction $v\in X$ and any real number $\alpha\ge 0$, one has
\begin{subequations}\label{formula0}
\begin{align}
\limsup_{x\to_v\xb} f^r(x;u+\alpha (\xb-x))&=
\limsup_{x\to_v\xb} f'(x;u+\alpha (\xb-x)) \label{formula0a}\\
&=\limsup_{x\to_v\xb}f^\del(x;u+\alpha (\xb-x)).
\label{formula0b}
\end{align}
\end{subequations}
\end{theorem}

\begin{proof}
\textit{First step.}
We claim that
\begin{equation}\label{formula00}
\limsup_{x\to_v\xb} f^\uparrow(x;u+\alpha (\xb-x))\le
\limsup_{x\to_v\xb} f^d(x;u+\alpha (\xb-x)).
\end{equation}
To prove this inequality, we take $\lambda \in\R$ such that
\begin{equation}\label{formula000}
\limsup_{x\to_v\xb} f^d(x;u+\alpha (\xb-x))<\lambda
\end{equation}
and show that $\lambda$ is greater than or equal to the left-hand side
of \eqref{formula00}.

From \eqref{formula000} we can find $\delta>0$ such that
\begin{equation}\label{formula01}
x\in D(\xb, v,\delta) \Rightarrow f^d(x;u+\alpha (\xb-x))<\lambda.
\end{equation}
Let $z=\xb+tv'\in D(\xb, v,\delta/2)$ and let
$\mu< f^\uparrow(z;u+\alpha(\xb-z))$.
By Theorem \ref{Treiman}, there exist $\eps>0$ and
$x\in B(z,\rho)$, with $0<\rho\le t\delta/2$ and $\alpha\rho\le\eps$,
such that
\begin{equation}\label{pas1}
\mu< f^d(x;u+\alpha(\xb-z)+w) \text{ for every } w\in B(0,\eps).
\end{equation}
Since $\alpha\|z-x\|\le \alpha\rho\le\eps$, putting $w=\alpha(z-x)$ in \eqref{pas1},
we infer that
\begin{equation}\label{pas2}
\mu< f^d(x;u+\alpha(\xb-x)).
\end{equation}
Since $\|x-\xb-tv'\|=\|x-z\|\le \rho\le t\delta/2$, we have $v^{''}:=(x-\xb)/t\in B(v',\delta/2)\subset B(v,\delta)$, showing that 
$x=\xb+tv^{''} \in D(\xb, v,\delta)$. Therefore, by \eqref{formula01},
\begin{equation}\label{pas3}
f^d(x;u+\alpha(\xb-x))<\lambda.
\end{equation}
Combining \eqref{pas2} and \eqref{pas3}, we derive that $\mu<\lambda$.
Since $\mu$ was arbitrarily chosen less than $f^\uparrow(z;u+\alpha(\xb-z))$,
we conclude that
$$z\in D(\xb, v,\delta/2) \Rightarrow
f^\uparrow(z;u+\alpha(\xb-z))<\lambda,$$
hence,
$$\limsup_{x\to_v\xb} f^\uparrow(x;u+\alpha (\xb-x))\le\lambda.$$
This completes the proof of \eqref{formula00}.

\medbreak
\textit{Second step.}
We claim that
\begin{equation}\label{formula00b}
\limsup_{x\to_v\xb} f^r(x;u+\alpha (\xb-x))\le
\limsup_{x\to_v\xb} f^\del(x;u+\alpha (\xb-x)).
\end{equation}
As in the first step, to prove this inequality we take
$\lambda \in\R$ such that
\begin{equation}\label{formula000b}
\limsup_{x\to_v\xb}f^\del(x;u+\alpha (\xb-x))<\lambda
\end{equation}
and show that $\lambda$ is greater than or equal to the left-hand side
of \eqref{formula00b}.

From \eqref{formula000b} we can find $\delta>0$ such that
\begin{equation}\label{formula02}
x\in D(\xb, v,\delta) \Rightarrow \sup \,\{ \la x^*, u+
\alpha (\xb-x)\ra\tq \xb^*\in\del f(x)\}<\lambda.
\end{equation}
Let 
$z=\xb+tv'\in D(\xb, v,\delta/2)$.
By Theorem \ref{JL}, for any $\mu< f^r(z;u+\alpha(\xb-z))$
and $\eps>0$ there exist
$x\in B(z,t\delta/2)$ and $x^*\in \del f(x)$
such that
$$\mu< \la x^*,u+\alpha(\xb-z)\ra \text{ and }
\la x^*,x-z\ra \le \eps.$$
As above, we can verify that $x\in D(\xb, v,\delta)$.
Therefore, by \eqref{formula02},
$$\mu< \la x^*,u+\alpha(\xb-z)\ra=\la x^*,u+\alpha(\xb-x)\ra
+\alpha \la x^*,x-z\ra < \lambda +\alpha\eps.$$
Since $\mu$ and $\eps$ were arbitrary, we derive that
$$z\in D(\xb, v,\delta/2) \Rightarrow f^r(z;u+\alpha(\xb-z))< \lambda,$$
showing that \eqref{formula00b} holds.

\medbreak
\textit{Third step.}
Since $\del f\subset \del_C f$, we have
$f^\del (z;u') \le f^{\del_C} (z;u')
\le f^\uparrow(z;u')$ for every $u'\in X$.
Hence, the right-hand side of \eqref{formula00b}
is less than or equal to the left-hand side of \eqref{formula00}.
On the other hand, $f^d\le f^r$. So all the expressions in
formulas \eqref{formula00} and \eqref{formula00b} are equal.
The desired set of equalities \eqref{formula0a}--\eqref{formula0b}
follows because $f^d\le f'\le f^\uparrow$.
\end{proof}

\begin{remark}
(a) In the special case $v=0$ and $\alpha= 0$, the formula \eqref{formula0a}
was proved by Borwein-Str\'ojwas \cite[Theorem 2.1 and Corollary 2.3]{BS89}
\smallbreak

(b) For $f$ locally Lipschitz at $\xb$, the formulas \eqref{formula0}
do not depend on $\alpha\ge 0$ since
\begin{equation*}\label{formula0lip}
\limsup_{x\to_v\xb} f^r(x;u+\alpha (\xb-x))=\limsup_{x\to_v\xb} f^r(x;u).
\end{equation*}
But they may depend on the direction $v\in X$:
for $f:x\in \R\mapsto f(x):=-|x|$ and $u\ne 0$, one has
$$\limsup_{x\to_u 0} f^r(x;u)=-|u|<\limsup_{x\to 0} f^r(x;u)=|u|.$$
\smallbreak

(c) For arbitrary lsc $f$, the value of the expressions
in \eqref{formula0} depends on $\alpha\ge 0$ even for convex $f$.
Indeed, as was recalled in Remark \ref{SP-rem}\,(a),
the graph of the subdifferential 
$$\del_{MR}f(\bx) = \{x^* \in X^* \tq \langle x^*,u\rangle \leq
f^r(\bx;u), \, \forall u \in X\}
$$
is generally not $strong\times weak^*$-closed.
\if{
in general the function $x\mapsto f^r(x;u)$ is not
upper semicontinuous; otherwise the subdifferential
$$\del_{MR}f(\bx) = \{x^* \in X^* \tq \langle x^*,u\rangle \leq
f^r(\bx;u), \, \forall u \in X\}.
$$
would have a $strong\times weak^*$-closed graph, which is not the case
(see, e.g. \cite[page 521]{JL02}). 
}\fi
Therefore, for an arbitrary lsc convex $f$ the function $x\mapsto f^r(x;u)$
is generally not upper semicontinuous, that is
$$f^r(\xb;u)<\limsup_{x\to\xb} f^r(x;u),$$
while always (see Proposition \ref{convexcase} below)
$$
f^r(\xb;u)=\inf_{\alpha\ge 0}\limsup_{x\to\xb} f^r(x;u+\alpha (\xb-x)).
$$
\end{remark}

\begin{proposition}[Radial subderivative for convex lsc
functions]\label{convexcase}
Let $X$ be a Banach space,
$f:X\to\xR$ be convex lsc, $\xb\in\dom f$ and $u\in X$.
Then,
\begin{equation}\label{convexformula0}
f^r(\xb;u)=\inf_{\alpha\ge 0}\limsup_{x\to\xb} f^r(x;u+\alpha (\xb-x)).\end{equation}
\end{proposition}
\begin{proof}
Of course, $f^r(\xb;u)$ is always not greater than the expression of the right-hand side
of \eqref{convexformula0}.
It is not smaller either since,
for every $t>0$,
$$
\limsup_{x\to\xb} f^r(x;tu+\xb-x)\le \limsup_{x\to\xb}\, (f(\xb+tu)-f(x))=f(\xb+tu)-f(\xb),
$$
hence, writing $\alpha=1/t$ for $t>0$,
\begin{align*}
\inf_{\alpha\ge 0}\limsup_{x\to\xb} f^r(x;u+\alpha (\xb-x))
&\le \inf_{t> 0}\limsup_{x\to\xb} \frac{1}{t} f^r(x;tu+\xb-x)\\
&\le \inf_{t> 0}\frac{f(\xb+tu)-f(\xb)}{t}=f^r(\xb;u). \tag*{\qedhere}
\end{align*}
\end{proof}

\section{Radially accessible functions}

A lsc function $f:X\to\xR$ is said to be
\textit{radially accessible} at $\xb\in\dom f$ from a direction $u\in X$
provided $$f(\xb)=\liminf_{t\searrow 0}f(\xb+tu),$$
or equivalently, provided
there exists a sequence $t_n\searrow 0$ such that 
$f(\xb+t_n u)\to f(\xb)$. (The case $u=0$ is a tautology.)

\smallbreak
\noindent \textit{Examples.}
1. Every lsc function $f:X\to\xR$ which is radially upper semicontinuous
at $\xb\in\dom f$ from $u$
is evidently radially accessible at $\xb$ from $u$.
This is the case of convex lsc
functions $f$ for any $u\in X$ such that
$\xb+u\in \dom f$.
\smallbreak
2. If $f^r(\xb;u)<\infty$, then $f$ is radially accessible
at $\xb$ from $u$.
Indeed, let $\gamma\in\R$ such that $f^r(\xb;u)<\gamma$.
Then, there exists $t_n\searrow 0$ such that
$f(\xb + t_n u)\le f(\xb)+\gamma t_n$,
and consequently
$\limsup_{n}f(\xb + t_n u)\le f(\xb)$.
The condition $f^r(\xb;u)<\infty$ however is not necessary:
the continuous function $f:\R\to\R$ given by
$f(x):=\sqrt{|x|}$ has $f^r(0;u)=\infty$ for any $u\ne 0$.
\smallbreak
3. The function $f:\R\to\R$ given by
$$f(x):=
\left\{
\begin{array}{ll}
0 & \mbox{if } x=0 \mbox{ or } x= 1/n, \mbox{ for }n=1,2,\ldots\\
1 & \mbox{otherwise}.
\end{array}
\right.
$$
is lsc on $\R$, not upper semicontinuous at $0$
along the ray $\R_+u$ for $u>0$ but 
radially accessible at $0$ from such $u>0$.
\smallbreak
4. The function $f:\R\to\R$ given by
$$
f(x):=
\left\{
\begin{array}{ll}
1 & \mbox{if } x>0\\
0 & \mbox{if } x\le 0
\end{array}
\right.
$$
is lsc on $\R$ but not radially accessible at $0$ from $u=1$.
We notice that $f^r(0;1)=+\infty$, while $f^r(x;1)=0$ for any $x>0$.
\smallbreak
Radial accessibility is a mild regularity property.
Yet this property leads to a more consistent behaviour
of subdifferentials and subderivatives.
We give two illustrations. Assume the lsc function $f$ is
radially accessible at $\xb\in\dom f$ from a direction $u$.
Then first,
$\dom \del f$ contains a sequence graphically and \textit{directionally}
convergent to $\xb$ (Theorem \ref{dense}),
and second,
the upper radial subderivative $f^r_+(\xb;u)$ is stable with respect
to radially convergent sequences (Proposition \ref{devdir}).
From the latter statement we derive that
the upper radial subderivative is a lower bound
for the expressions in \eqref{formula0} with $v=u$ (Theorem \ref{belowap}).

\smallbreak
We recall the statement of Ekeland's variational principle \cite{Eke74}:
\ps
\textit{Variational Principle}.
For any lsc function $f$ defined on a closed subset $S$ of a Banach space,
$\bx\in\dom f$ and $\eps>0$ such that
$
f(\xb)\le \inf f(S) +\eps,
$
and for any $\lambda>0$, there exists $x_\lambda\in S$ such that
$\|x_\lambda-\xb\|\le\lambda$, $f(x_\lambda)\le f(\xb)$ and
the function $x\mapsto f(x)+(\eps/\lambda)\|x-x_\lambda\|$
attains its minimum on $S$ at $x_\lambda$.
\begin{theorem}[Directional density of subdifferentials]\label{dense}
Let $X$ be a Banach space,
$f:X\to\xR$ be lsc, $\xb\in\dom f$ and $u\in X$
such that $f$ is radially accessible at $\xb$ from $u$.
Then, there exists a sequence $((x_n,x_n^*))_n\subset\del f$ such that
$x_n\to_u \xb$, $f(x_n)\to f(\xb)$ and $\limsup_n \la x_n^*, x_n-\bx\ra\le 0$.
\end{theorem}

\begin{proof}
Let $\eps>0$.
Since $f$ is \lsc\ at $\xb$, there exists $\delta\in ]0,\eps^2[$ such that
\begin{equation}\label{sci0}
f(\xb)\le \inf f(B_\delta(\xb)) + \eps^2,
\end{equation}
and since $f(\xb)=\liminf_{t\searrow 0}f(\xb+tu)$,
there exists $\mu>0$ such that $\mu(\eps+\|u\|)<\delta$ and
$f(\xb+\mu u)\le f(\xb)+\eps^2$.
Summarizing, we can find real numbers $\delta$ and $\mu$ satisfying
\begin{subequations}\label{sci}
\begin{gather}
0< \mu(\eps+\|u\|)<\delta<\eps^2, \mbox{ and} \label{sci1}\\
f(\xb+\mu u)\le \inf f(B_\delta(\xb)) + \eps^2\label{sci2}.
\end{gather}
\end{subequations}
Now we apply Ekeland's variational principle to $f$ on the set $B_\delta(\xb)$
at point $\xb+\mu u$ with $\lambda=\mu\eps$. Observe that the ball
$B_\lambda(\xb+\mu u)$ is contained in the ball $B_\delta(\xb)$
by \eqref{sci1}. We therefore obtain a point $\xt\in X$ such that
\begin{subequations}
\begin{gather}
\|\xt-(\xb+\mu u)\|< \mu\eps, ~f(\xt)\le f(\xb+\mu u), \mbox{ and}  \label{BP1}\\
y\mapsto f(y)+(\eps/\mu)\|y-\xt\| \mbox{ admits a local minimum at } \xt. \label{BP2}
\end{gather}
\end{subequations}
In view of (\ref{BP2}), we may apply the Separation Principle at point $\xt$
with the convex Lipschitz function $\varphi:y\mapsto (\eps/\mu)\|y-\xt\|$ to obtain
a subgradient $\xt^*\in \delc f(\xt)$ such that
\begin{equation}\label{BP3}
\|\xt^*\|\le \eps/\mu.
\end{equation}
Now, take $(\bx_\eps,\bx_\eps^*)\in \del f$ such that
\begin{subequations}
\begin{gather}
\|\bx_\eps-x_\eps\|<\mu\eps,~~|f(\xb_\eps)-f(x_\eps)|<\eps^2, \mbox{ and}  \label{clos1}\\
\la\bx_\eps^*-x_\eps^*,\bx-x_\eps\ra>-\eps,~~\la\bx_\eps^*,\bx_\eps-x_\eps\ra<\eps.
\label{clos2}
\end{gather}
\end{subequations}
It follows from the first parts of \eqref{BP1} and \eqref{clos1} that
\begin{equation}\label{cond1}
\|\bx_\eps-(\xb+\mu u)\|< 2\mu\eps,
\end{equation}
from the second parts of \eqref{BP1} and \eqref{clos1}
combined with \eqref{sci0} and \eqref{sci2} that 
\begin{equation}\label{cond2}
|f(\bx_\eps)-f(\xb)|\le 2\eps^2,
\end{equation}
and from \eqref{BP3} and \eqref{clos2} that
\begin{align}\label{cond3}
\la \bx_\eps^*, \bx-\bx_\eps\ra >-\eps(\eps+\|u\|+2),
\end{align}
since
\begin{align*}
\la \bx_\eps^*, \bx-\bx_\eps\ra 
            &= \la x_\eps^*, \bx-x_\eps\ra+
               \la \bx_\eps^*-x_\eps^*, \bx-x_\eps\ra + 
               \la \bx_\eps^*, x_\eps-\bx_\eps\ra\\
            &>  -\|\xt^*\|\|\xt-\xb\|-2\eps\\
            &> -(\eps/\mu)\mu\eps+\mu \|u\|-2\eps=-\eps(\eps+\|u\|+2).
\end{align*}
Therefore, if for every $n\in\N$, we let $\eps=1/n$
and choose $\mu_n=\mu$ satisfying \eqref{sci1},
so that $0<\mu_n<1/n$, 
we obtain a sequence $((x_n,x_n^*))_n$ in $\del f$
by setting $x_n:=\bx_\eps$ and $x_n^*:=\bx_\eps^*$.
It follows from \eqref{cond1}, \eqref{cond2} and \eqref{cond3}
that this sequence satisfies the requirements of the theorem.
\end{proof}

\begin{remark}
(a) The case $u=0$ in Theorem \ref{dense} is known, see, e.g.
\cite{JL13,JL14}.

\ps (b) The case $u\ne 0$ is new even for convex lsc functions (recall
that such functions are radially accessible at any point $\xb\in\dom f$
from any $u$ such that $\xb+u\in\dom f$).
\end{remark}

\smallbreak
(c) For $u\ne 0$, the conclusion of Theorem \ref{dense}
can be false at points where the function is not
radially accessible. Let $f:\R\to\xR$ given by
$$f(x):=
\left\{
\begin{array}{ll}
0 & \mbox{if } x=0 \mbox{ or }  x= 1/n, \mbox{ for }n=1,2,\ldots\\
+\infty & \mbox{otherwise}.
\end{array}
\right.
$$
Then, $f$ is lsc on $\R$ but not radially accessible at
any point $\xb=1/n$ from $u\ne 0$.
We observe that all the points $x\ne\xb$ close to $\xb$
are not in $\dom f$, hence $\del f(x)=\emptyset$. 
\smallbreak
(d) For $u\ne 0$, we cannot claim in the conclusion of Theorem \ref{dense}
to find a radially convergent sequence
$(x_n)$ instead of a directionally convergent one. Consider
the function $f:\R^2\to\xR$ given, for $x=(\xi_1,\xi_2)$,  by
$$f(x):=
\left\{
\begin{array}{ll}
-\sqrt{\xi_1}& \mbox{if } (\xi_1, \xi_2)\in \R_+\times \R \\
+\infty  & \mbox{otherwise}.
\end{array}
\right.
$$
Then, $f$ is convex lsc on $\R^2$ and $f(0,t)=0$ for every $t\in\R$, so $f$ is
radially continuous at $\xb=(0,0)$ in the direction $u=(0,1)$.
But, for every $t\in\R$ we have $\del f(0,t)=\emptyset$,
so there is no sequence $(x_n)$ radially convergent to $\xb$
from the direction $u=(0,1)$ with $\del f(x_n)\not=\emptyset$.
\medbreak
We recall the statement of the mean value inequality using the radial subderivative
\cite{JL13,JL14}:

\begin{lemma}[Basic mean value inequality]\label{mvi}
Let $X$ be a Hausdorff locally convex space, $f:X\to\xR$ be lsc,
$\xb\in X$ and $x\in\dom f$.
Then, for every real number $\lambda\le f(\xb)-f(x)$, there
exist $t_0\in [0,1[$ and $x_0:=x+t_0(\xb-x)\in [x,\xb[$
such that $f(x_0)\le f(x)+t_0\lambda$ and\smallskip\\
\centerline{ $\lambda\le f^r(x_0;\xb-x).$}
\end{lemma}

\begin{proposition}[Radial stability of the upper radial subderivative]\label{devdir}
Let $X$ be a Hausdorff locally convex space, $f:X\to\xR$ be lsc,
$\xb\in\dom f$ and $u\in X$
such that $f$ is radially accessible at $\xb$ from $u$.
Then, there is a sequence
$\mu_n\searrow 0$ such that $f(\xb+\mu_n u)\to f(\xb)$ and
\begin{equation}\label{applimvi1}
f^r_+(\xb;u)\le \liminf_{n\to +\infty}f^r(\xb+\mu_n u;u).
\end{equation}
In particular,
\begin{equation}\label{below00}
f^r_+(\xb;u)\le \inf_{\alpha\ge 0}
\limsup_{x\to_u\xb} f^r(x;u+\alpha (\xb-x)).
\end{equation}
\end{proposition}

\begin{proof}
If $f^r_+(\xb;u)=-\infty$, there is nothing to prove.
Otherwise, it suffices to show that for
every $\lambda< f^r_+(\xb;u)$, there exists a sequence
$\mu_n\searrow 0$ such that $f(\xb+\mu_n u)\to f(\xb)$ and
\begin{equation}\label{applimvi1b}
\ld\le \liminf_{n\to +\infty}f^r(\xb+\mu_n u;u).
\end{equation}
So, let $\ld< f^r_+(\xb;u)$. 
By definition of $f^r_+(\xb;u)$ there is a sequence $\tau_n\searrow 0$ such that
\begin{equation*}
\ld \tau_n < f(\xb+\tau_nu)-f(\xb)\quad\mbox{for every } n\in \N,
\end{equation*}
and by assumption, there is a sequence $t_n\searrow 0$ such that
$f(\xb+t_n u)\to f(\xb)$.
For every $n$, let $k_n\in \N$ such that
\begin{subequations}\label{dd1}
\begin{gather}
\ld \tau_n < f(\xb+\tau_nu)-f(\xb+ t_{k_n}u),\label{dd1a}\\
 0<t_{k_n}<\tau_n^2.\label{dd1b}
\end{gather}
\end{subequations}
Applying Lemma \ref{mvi} to \eqref{dd1a}
we obtain $t_0\in [0,1[$ and $x_0:=\xb+t_{k_n}u+t_0(\tau_n-t_{k_n})u=\xb+\mu_n u$
with $\mu_n:=t_{k_n}+t_0(\tau_n-t_{k_n})\in [t_{k_n},\tau_n[$ such that
\begin{subequations}\label{dd2}
\begin{gather}
f(\xb+\mu_n u)\le f(\xb+t_{k_n}u)+t_0\ld \tau_n, \label{dd2a}\\
\ld \tau_n/(\tau_n-t_{k_n})< f^r(\xb+\mu_n u;u). \label{dd2b}
\end{gather}
\end{subequations}
Observe that in view of \eqref{dd1b} one has $\tau_n/(\tau_n-t_{k_n})\to 1$,
so, letting $n\to +\infty$ we obtain
\begin{equation*}
\mu_n\searrow 0,\quad
\limsup_{n\to +\infty} f(\xb+\mu_n u) \le f(\xb),\quad
\ld \le  \liminf_{n\to +\infty}f^r(\xb+\mu_n u;u).
\end{equation*}
This completes the proof of the first statement since
we also have $$f(\xb)\le \liminf_{n\to +\infty} f(\xb+\mu_n u)$$
by the lower semicontinuity of $f$ at $\xb$.
\smallbreak
To show the second statement, let $x_n:=\xb+\mu_n u$
with $\mu_n\searrow 0$ such that 
\begin{equation}\label{z1}
f^r_+(\xb;u)\le \liminf_{n\to +\infty}f^r(x_n;u).
\end{equation}
Since $u+\alpha (\xb-x_n)=(1-\alpha\mu_n)u$, it follows that,
for any $\alpha\ge 0$,
\begin{equation}\label{z2}
\liminf_{n\to +\infty}f^r(x_n;u+\alpha (\xb-x_n))
=\liminf_{n\to +\infty}\,(1-\alpha\mu_n)f^r(x_n;u)=\liminf_{n\to +\infty}f^r(x_n;u).
\end{equation}
Since $x_n\to_u \xb$,
we derive from \eqref{z1} and \eqref{z2} that, for every $\alpha\ge 0$,
\begin{align*}
f^r_+(\xb;u)\le \liminf_{n\to +\infty}f^r(x_n;u+\alpha (\xb-x_n))\le
\limsup_{x\to_u\xb} f^r(x;u+\alpha (\xb-x)). \tag*{\qedhere}
\end{align*}
\end{proof}

Plugging the formula \eqref{formula0b} of Theorem \ref{formula}
into the inequality \eqref{below00} of Proposition \ref{devdir}, we
immediately obtain that
the upper radial subderivative $f^r_+(\xb;u)$ is a lower bound
for the directional limit superior of the support of any subdifferential:

\begin{theorem}[Link between upper radial subderivative and
subdifferentials]\label{belowap}
Let $X$ be a Banach space,
$f:X\to\xR$ be lsc, $\xb\in\dom f$ and $u\in X$
such that $f$ is radially accessible at $\xb$ from $u$.
Then,
\begin{equation}\label{below000}
f^r_+(\xb;u)\le\inf_{\alpha\ge 0}
\limsup_{x\to_u\xb}\,\delsf(x;u+\alpha (\xb-x)).
\end{equation}
\end{theorem}

\section{Appendix: a direct proof of Theorem \ref{belowap}}

For the sake of completeness, we provide a direct proof
Theorem \ref{belowap}. In fact, we shall establish an inequality more accurate
than \eqref{below000}, in the same vein as Theorem \ref{JL} and
Theorem \ref{dense}.

\begin{theorem}[Refined link between upper radial subderivative and
subdifferentials]\label{belowapbis}
Let $X$ be a Banach space,
$f:X\to\xR$ be lsc, $\xb\in\dom f$ and $u\in X$
such that $f$ is radially accessible at $\xb$ from $u$.
Then, there is a sequence $((x_n,x_n^*))\subset\del f$ such that
$x_n\to_u \xb$, $f(x_n)\to f(\xb)$
and
\begin{equation}\label{below0}
f^r_+(\xb;u)\le \liminf_{n}\,\langle x^*_n,u+\alpha(\xb-x_n)\rangle, \quad
\forall \alpha\ge 0.
\end{equation}
\end{theorem}

\begin{proof}
The pattern of the proof is similar to that of 
\cite[Theorem 2.1]{JL14}
but the argument has to be refined in order to obtain a directionally
convergent sequence $(x_n)$.
\ps
\textit{First step.}
If $u=0$ or if $f^r_+(\xb;u)=-\infty$, the result follows from
Theorem \ref{dense}.
Otherwise, assume $u\ne 0$, let $\gamma<f^r_+(\xb;u)$ and let $\eps>0$.
We claim that for each $n\in\N$ sufficiently large,
there exists $(x_n,x_n^*)\in \delc f$ such that
\begin{subequations}\label{undeux}
\begin{gather}
x_{n}\in D(\xb,u,\eps), \quad f(x_{n})< f(\xb)+\eps,
\label{un}\\
\langle x^*_n,u+\alpha(\xb-x_n)\rangle>\gamma-(\alpha+1)\eps, \quad
\forall \alpha\ge 0.\label{deux}
\end{gather}
\end{subequations}
Let $z^*\in X^*$ such that 
$\la z^*,u\ra =-\gamma$, set $g:=f+z^*$ and let $K:= [\xb, \xb+u]$.
Let also $0<\delta<1$ such that $g$ is bounded below on $B_\delta(K)$.
By Proposition \ref{devdir}, there exists a sequence $\mu_n\searrow 0$ such that
\begin{subequations}\label{etoile0}
\begin{gather}
|f(\xb+\mu_n u)- f(\xb)|<1/n, \label{etoile0a}\\
\gamma< f^r(\xb+\mu_n u;u).\label{etoile0b}
\end{gather}
\end{subequations}
We may assume $0<\mu_n< \sqrt{\delta}$.
By \eqref{etoile0b}, there exists $t_n\in {]}0,1-\mu_n]$ such that
\begin{equation}\label{etoile}
f(\xb+\mu_n u)\le f(\xb+\mu_n u+tu) -\gamma t, \quad \forall  t\in [0,t_n].
\end{equation}
Let $K_n:=[\xb+\mu_n u,\xb+(\mu_n +t_n)u]\subset K$. 
Then, (\ref{etoile}) can be rewritten as
\begin{equation}\label{etoile2}
g(\xb+\mu_n u) \le g(x), \quad \forall x \in K_n.
\end{equation}
Take $r>0$ such that 
\begin{equation}\label{aa}
g(\xb+\mu_n u)<\inf_{B_{r}(K_n)}g+\mu_n^3t_n,
\end{equation}
and, observing that both $\inf_{B_{r}(K_n)}g$ and $\inf_{B_{\delta}(K_n)}g$ are finite, choose $\alpha_n>0$ such that
\begin{equation*}\label{aaa}
\inf_{B_{r}(K_n)}g\le \inf_{B_{\delta}(K_n)}g+\alpha_n r^2.
\end{equation*}
Then
\begin{equation*}\label{aaaa}
\inf_{B_{r}(K_n)}g \le (g +\alpha_n d^2_{K_n})(x), \quad \forall x\in B_{\delta}(K_n),
\end{equation*}
and therefore, by \eqref{aa},
\begin{equation}\label{penal}
g(\xb+\mu_n u) \le (g +\alpha_n d^2_{K_n})(x) + \mu_n^3t_n,
\quad \forall x\in B_{\delta}(K_n).
\end{equation}
Now, apply Ekeland's variational principle to the function
$g +\alpha_n d^2_{K_n}$ on the set $B_\delta(K_n)$
at point $\xb+\mu_n u\in K_n$ with $\eps=\mu_n^3t_n$ and
$\lambda = \mu_n^2t_n$.
Observe that the ball $B_\lambda(\xb+\mu_n u)$ is contained in
$B_{\delta}(K_n)$ since for every $x\in B_\lambda(\xb+\mu_n u)$, we have
$d_{K_n}(x)\le \|x-(\xb+\mu_n u)\|\le \lambda= \mu_n^2t_n<\delta$.
We then obtain a point $x_n\in X$ satisfying
\begin{subequations}\label{reBP}
\begin{gather}
\|x_n-(\xb+\mu_n u)\|< \mu_n^2t_n,~~ g(x_n)+\alpha_n d^2_{K_n}(x_n)\le g(\xb+\mu_n u)
\label{reBP1}\\
y\mapsto f(y)+\la z^*,y\ra+\alpha_n d^2_{K_n}(y)+\mu_n\|y-x_n\| \mbox{ admits a local minimum at } x_n.
\label{reBP2}
\end{gather}
\end{subequations}
It follows from the first half of \eqref{reBP1} that
\begin{equation*}
x_n-\xb\in B(\mu_n u,\mu_n^2t_n)=\mu_n B(u,\mu_nt_n),
\end{equation*}
showing that $x_{n}\in D(\xb,u,\eps)$ for $n$ sufficiently large.
On the other hand, the second half of \eqref{reBP1} and \eqref{etoile0a}
entail
\begin{equation*}
f(x_n)\le f(\xb+\mu_n u)+\|z^*\| \|\xb+\mu_n u-x_n\|
\le f(\xb)+\|z^*\| \mu_n^2t_n+1/n,
\end{equation*}
showing that $f(x_n)< f(\xb)+\eps$  for $n$ sufficiently large.
\ps
In view of (\ref{reBP2}), we may apply the Separation Principle at point $x_n$
with the convex Lipschitz function
$\varphi:y\mapsto \la z^*,y\ra+\alpha_n d^2_{K_n}(y)+\mu_n\|y-x_n\|$
to obtain points
$x_n^*\in \delc f(x_n)$, $\zeta_n^*\in\del d^2_{K_n}(x_n)$ and $\beta_n^*\in B^*$  with
\begin{equation}\label{recond}
0=x_n^*+z^*+\alpha_n \zeta_n^*+\mu_n \beta_n^*.
\end{equation}
We claim that the pair $(x_n,x_n^*)\in \delc f$
satisfies \eqref{deux} for large $n\in\N$.
Assume it can be shown that
for all $\alpha\ge 0$ and for large $n\in\N$,
\begin{equation}\label{ff0}
\langle \zeta^*_n,u+\alpha(\xb-x_n)\rangle\le 0.
\end{equation}
Then, it follows that for all $\alpha\ge 0$ and for large $n\in\N$,
\begin{align*}
\la x_n^*,u+\alpha(\xb-x_n)\ra &=\la-z^*,u+\alpha(\xb-x_n)\ra
-\alpha_n\la\zeta_n^*,u+\alpha(\xb-x_n)\ra-2\mu_n\la \beta_n^*,u+
\alpha(\xb-x_n)\ra \\
&\ge\gamma-\alpha\|z^*\|\|\xb-x_n\|-2\mu_n\|u+\alpha(\xb-x_n)\|,
\end{align*}
which implies that
$\langle x^*_n,u+\alpha(\xb-x_n)\rangle>\gamma-(\alpha+1)\eps$
for $n$ sufficiently large, as claimed.
\ps
\textit{Second step.}
To complete the proof of \eqref{deux} it remains to prove \eqref{ff0}.
We first consider the case $\alpha=0$, that is, we show
\begin{equation}\label{claim2}
\la \zeta_n^*, u\ra \le 0, \quad
\forall n\in\N.
\end{equation}
Let $P_{K_n}{x_n}\in K_n$ be any point such that $\|x_n -P_{K_n}{x_n}\|= d_{K_n}(x_n)$.
We have
$$\|\xb + \mu_nu-P_{K_n}{x_n}\|\le \|\xb + \mu_nu-x_n\|+\|x_n-P_{K_n}{x_n}\|<2\mu_n^2t_n.$$ 
So, $P_{K_n}{x_n}=\xb + \mu_nu+\tau_nu$ with $\tau_n\|u\|<2\mu_n^2t_n$.
Hence, $t_n-\tau_n>0$ for large $n$, and
$$
(t_n-\tau_n)u=\xb + (\mu_n+t_n)u-P_{K_n}{x_n}.
$$
Notice that $\zeta_n^*=2 d_{K_n}(x_n) \xi_n^*$ where $\xi_n^*\in\del d_{K_n}(x_n)$.
Then:
\begin{align*}
(t_n-\tau_n)\la\zeta_n^*,u\ra &=\la \zeta_n^*,\xb+(\mu_n+t_n) u-P_{K_n}{x_n}\ra\\
&=\la \zeta_n^*,\xb+(\mu_n+t_n)u-x_n\ra+ \la\zeta_n^*,x_n-P_{K_n}{x_n}\ra\\
&=  2 d_{K_n}(x_n) \left(\la\xi_n^*,\xb+(\mu_n+t_n)u-x_n\ra+ \la\xi_n^*,x_n-P_{K_n}{x_n}\ra\right)\\
&\le  2 d_{K_n}(x_n) (- d_{K_n}(x_n) + \|x_n -P_{K_n}{x_n}\|)=0.
\end{align*}
This proves \eqref{claim2}.

Now consider the case $\alpha> 0$. Write $\alpha=1/t$. 
We must show that for large $n\in\N$,
\begin{equation}\label{claim1}
\langle \zeta^*_n,u+\alpha(\xb-x_n)\rangle=(1/t)\la \zeta_n^*, \xb+t u-x_n\ra \le 0.
\end{equation}
But, for $n$ so large that $\mu_n<t$, we have
\begin{align*}
\la \zeta_n^*, \xb+t u-x_n\ra &= 
\la \zeta_n^*, \xb+\mu_n u-x_n\ra +\la \zeta_n^*, (t-\mu_n) u\ra\\
&\le d^2_{K_n}(\xb+\mu_n u)-d^2_{K_n}(x_n)+(t-\mu_n)\la \zeta_n^*, u\ra\\
&\le -d^2_{K_n}(x_n)\quad \mbox{by } \eqref{claim2}\\
&\le 0.
\end{align*}
This proves \eqref{claim1}. Hence, \eqref{ff0} holds and so also
\eqref{deux}, as we have observed.
\ps
\textit{Third step.}
Every pair $(x_n,x_n^*)$ in $\delc f$ is close to a
pair $(\xb_n,\xb_n^*)$ in $\del f$ in such a way that
the sequence $((x_n,x_n^*))_n$ satisfying \eqref{un}--\eqref{deux}
for large $n\in\N$ can actually be assumed to lie in $\del f$
(proceed as in Theorem \ref{dense}).

\smallbreak
\textit{Fourth step.}
The theorem is derived from \eqref{un}--\eqref{deux} as follows.
Let $(\gamma_k)_k$ be an increasing sequence
of real numbers such that $\gamma_k\nearrow f^r(\xb;d)$.
We have proved that, for each $k\in\N$, there
are a sequence $((x_{n,k},x_{n,k}^*))_n
\subset\del f$ and an integer $N_k\in\N$ satisfying for every $n\ge N_k$:
\begin{subequations}
\begin{gather}
x_{n,k}\in D(\xb,u,1/k), \quad f(x_{n,k})< f(\xb)+ 1/k,
\label{unfin}\\
\langle x^*_{n,k},u+\alpha(\xb-x_{n,k})\rangle
>\gamma_k-(\alpha+1)/k, \quad
\forall \alpha\ge 0.\label{deuxfin}
\end{gather}
\end{subequations}
Clearly, we may assume $N_{k+1}>N_k$. Then,
it is immediate from \eqref{unfin}--\eqref{deuxfin}
that the diagonal sequence defined, for $k\in\N$,
by $(x_k,x_k^*) := (x_{N_k,k},x_{N_k,k}^*)$ satisfies
the assertions of the theorem.
\end{proof}

{\small

}
\end{document}